\documentclass[11pt, reqno]{amsart}
\usepackage[utf8]{inputenc}
\usepackage{amsmath,bm}
\usepackage{amsthm}
\usepackage{amsfonts}
\usepackage{MnSymbol}
\usepackage{blindtext,amsmath,amsthm,amsfonts, textcomp, mathrsfs, mathtools}
\usepackage[shortlabels]{enumitem}

\usepackage[dvipsnames]{xcolor}
\definecolor{Red}{cmyk}{0, 0.7808, 0.4429, 0.1412}

\usepackage[top=30mm,bottom=30mm,left=30mm,right=30mm]{geometry}
\newtheorem{theorem}{Theorem}[section]

\newtheorem{corollary}{Corollary}

\newtheorem*{theorem*}{Theorem}
\newtheorem*{remark*}{Remark}
\newtheorem*{problem*}{Problem}
\newtheorem*{conjecture*}{Conjecture}
\newtheorem*{question*}{Question}
\newtheorem{lemma}[theorem]{Lemma}

\newtheorem{definition}[theorem]{Definition}

\newcommand{\rom}[1]{\uppercase\expandafter{\romannumeral #1\relax}}

\begin{document}

\title[Large values of $L$-functions on $1$-line]{Large values of $L$-functions on $1$-line}

\author[Anup B. Dixit]{Anup B. Dixit}
\address{Department of Mathematics and Statistics\\ Queen's University\\ Jeffrey Hall, 48 University Ave\\ Kingston\\ Canada, ON\\ K7L 3N8}
\email{anup.dixit@queensu.ca}

\author[Kamalaskhya Mahatab]{Kamalakshya Mahatab}
\address{Kamalakshya Mahatab, Department of Mathematics and Statistics, University of Helsinki, P. O. Box 68, FIN
00014 Helsinki, Finland}
\email{accessing.infinity@gmail.com, \. kamalakshya.mahatab@helsinki.fi}

\thanks{ABD is supported by the Coleman postdoctoral fellowship of Queen's university}
\thanks{KM is supported by Grant 227768 of the Research Council of Norway and Project 1309940 of Finnish Academy.} 

\date{}

\begin{abstract} 
In this paper, we study lower bounds of a general family of $L$-functions on the
$1$-line. More precisely, we show that for any $F(s)$ in this family, there exists arbitrary large $t$
such that $F(1+it)\geq e^{\gamma_F} (\log_2 t + \log_3 t)^m + O(1)$, where $m$ is the order of the pole of $F(s)$ at
$s=1$. This is a generalization of the same result of Aistleitner, Munsch and the second author for the Riemann zeta-function. As a consequence, we get lower bounds for large values of Dedekind zeta-functions and Rankin-Selberg $L$-functions of the type $L(s,f\times f)$ on the $1$-line.
\end{abstract}



\maketitle
\section{\bf Introduction}
The growth of the Riemann zeta-function $\zeta(s)$ in the critical strip $1/2<\Re(s)<1$ has been of interest to number theorists for a long time. In this context, the upper bound is predicted by the Lindel\"of hypothesis, which claims that $|\zeta(\sigma+it)| \ll |t|^{\epsilon}$ for any $\epsilon >0$ and $1/2 < \sigma<1$. This is, in fact a consequence of the famous Riemann hypothesis. Although there is significant progress towards this problem, no unconditional proof is known (see \cite{Titchmarsh} for more details). \\

A more intricate question is to investigate how large $|\zeta(\sigma+it)|$ can be for a fixed $\sigma \in [1/2,1)$ and $t\in [T,T+H]$, some interval. In this direction,  Balasubramanian and Ramachandra \cite{BaluRamch} showed that 
\begin{equation*}
    \max_{t\in [T,T+H]} \left|\zeta\left(\frac{1}{2}+it\right)\right| \geq \exp\left( c  \sqrt{\frac{\log H}{\log_2 H}}\right),
\end{equation*}
where $c$ is a positive constant, $H\ll \log_2 T$ and $\log_2 T$ denotes $\log\log T$. From now on, we will denote $\underbrace{\log\log\ldots\log T}_{k\text{ times }}$ by $\log_k T$. This result was improved by Bondarenko and Seip \cite{BondSeip} in a larger interval and was later optimized by Bretche and Tenenbaum \cite{BreTen}, who showed that
\begin{equation*}
    \max_{t\in [0,T]} \left|\zeta\left(\frac{1}{2}+it\right)\right| \geq \exp\left( (\sqrt 2 + o(1))  \sqrt{\frac{\log T\log_3 T}{\log_2 T}}\right).
\end{equation*}
For $\sigma \in (1/2, 1)$ and $c_\sigma=0.18(2\sigma-1)^{1-\sigma}$, Aistleitner \cite{aist} proved that 
\begin{equation*}
    \max_{t\in [0,T]} \left|\zeta\left(\frac{1}{2}+it\right)\right| \geq \exp\left( c_\sigma  \sqrt{\frac{\log T}{\log_2 T}}\right).
\end{equation*}

On the other hand, we expect much finer results for large values of $L$-functions on $\Re(s)=1$. In \cite{Sound}, Granville and Soundararajan used techniques of diophantine approximation to show that
\begin{equation*}
    \max_{t\in[0,T]} |\zeta(1+it)| \geq e^{\gamma} \bigg(\log_2 T + \log_3 T - \log_4 T + O(1)\bigg)
\end{equation*}
for arbitrarily large $T$. This is an improvement on the previous bounds given by Levinson \cite{Lev}. Granville and Soundararajan \cite{Sound} conjectured that
\begin{equation}\label{sound-conjecture}
    \max_{t\in[T,2T]} |\zeta(1+it)| = e^{\gamma} (\log_2 T + \log_3 T + C_1 + o(1)),
\end{equation}
where $C_1$ is an explicitly computable constant. \\

In 2017, Aistleitner, Munsch and the second author \cite{Chris} used the resonance method to prove that there is a constant $C$ such that
\begin{equation}\label{zeta-large_values}
    \max_{t\in[\sqrt{T},T]} |\zeta(1+it)| \geq e^{\gamma} (\log_2 T + \log_3 T + C).
\end{equation}
Note that this result essentially matches \eqref{sound-conjecture}, however, the size of the interval here is much larger. Unfortunately, over shorter intervals $[T,T+H]$, very little seems to be known regarding large values of $\zeta(1+it)$ (see \cite{Maksym}, \cite{ouimet} for further details).\\


In this paper, we generalize \eqref{zeta-large_values} to a large class of $L$-functions, namely $\mathbb{G}$, which conjecturally contains the Selberg class $\mathbb{S}$. We establish \eqref{zeta-large_values} for elements in $\mathbb{G}$ with non-negative Dirichlet coefficients. The key difference between $\mathbb{G}$ and $\mathbb{S}$ is that elements in $\mathbb{G}$ satisfy a polynomial Euler-product which is a more restrictive condition than that in $\mathbb{S}$. However, the functional equation in $\mathbb{S}$ is replaced by a weaker ``growth condition" in $\mathbb{G}$. This is a significant generalization because most Euler products, which have an analytic continuation exhibit a growth condition, but perhaps not a functional equation. As applications, we prove the analogue of \eqref{zeta-large_values} for Dedekind zeta-functions $\zeta_K(s)$ and Rankin-Selberg $L$-functions given by $L(s,f\times f)$. We also prove a generalized Merten's theorem for $\mathbb{G}$ as a precursor to the proof of our main theorem.\\

The resonance method with a similar resonator was used by Aistleitner, Munsch, Peyrot and the second author \cite{Chris3} to establish large values of Dirichlet $L$-functions $L(s,\chi)$ with a given conductor $q$ at $s=1$. Perhaps, a similar method can also be used to establish large values over more general orthogonal families of $L$-functions in $\mathbb{G}$.\\

\subsection{\bf The class $\mathbb{G}$}
In 1989, Selberg \cite{Slb} introduced a class of $L$-functions $\mathbb{S}$, which is expected to encapsulate all naturally occurring $L$-functions arising from arithmetic and geometry.
\begin{definition}[The Selberg class]
The Selberg class $\mathbb{S}$ consists of meromorphic functions $F(s)$ satisfying the following properties.
\begin{enumerate}
    \item[(i)] {\bf Dirichlet series} - It can be expressed as a Dirichlet series
    \begin{equation*}
         F(s) = \sum_{n=1}^{\infty} \frac{a_F(n)}{n^s},
    \end{equation*}
    which is absolutely convergent in the region $\Re(s)>1$. We also normalize the leading coefficient as $a_F(1)=1$.
    
    \item[(ii)] {\bf Analytic continuation} - There exists a non-negative integer $k$, such that $(s-1)^k F(s)$ is an entire function of finite order.
    
    \item[(iii)] {\bf Functional equation} - There exist real numbers $Q>0$ and $\alpha_i\geq 0$, complex numbers $\beta_i$ and $w\in\mathbb{C}$, with $\Re(\beta_i) \geq 0$ and $|w| =1$, such that
    \begin{equation}\label{fneq}
        \Phi(s) := Q^s \prod_i \Gamma(\alpha_i s + \beta_i) F(s)
    \end{equation}
    satisfies the functional equation
    \begin{equation*}
        \Phi(s) = w \overline{\Phi}(1-\overline{s}).
    \end{equation*}

    \item[(iv)] {\bf Euler product} - There is an Euler product of the form
    \begin{equation}\label{eprod}
        F(s) = \prod_{p \text{ prime}} F_p(s),
    \end{equation}
    where
    \begin{equation*}
        \log F_p(s) = \sum_{k=1}^\infty \frac{b_{p^k}}{p^{ks}}
    \end{equation*}
    with $b_{p^k} = O(p^{k \theta})$ for some $\theta < 1/2$.
    
    \item[(v)] {\bf Ramanujan hypothesis} - For any $\epsilon >0$, 
    \begin{equation}\label{rhyp}
        |a_F(n)| = O_\epsilon(n^\epsilon).
    \end{equation}
\end{enumerate}
\end{definition}

The Euler product implies that the coefficients $a_F(n)$ are multiplicative, i.e., $a_F(mn) = a_F(m) a_F(n)$ when $(m,n)=1$. Moreover, each Euler factor also has a Dirichlet series representation
\begin{equation*}
    F_p(s) = \sum_{k=0}^{\infty}\frac{a_F(p^k)}{p^{ks}},
\end{equation*}
which is absolutely convergent on $\Re(s)>0$ and non-vanishing on $\Re(s)> \theta$, where $\theta$ is as in $(iv)$.\\

For the purpose of this paper, we need a stronger Euler-product to ensure that the Euler factors factorize completely and further require a zero free region near $1$-line, similar to what we notice in the proof of  prime number theorem. However, we can replace the functional equation with a weaker condition on the growth of $L$-functions on vertical lines. This leads to the definition of the class $\mathbb{G}$.
\begin{definition}[The class $\mathbb{G}$]\label{def:S_ha}
The class $\mathbb{G}$ consists of meromorphic functions $F(s)$ satisfying $\textit{(i), (ii)}$ as in the above definition and further satisfies
\begin{enumerate}
    \item[(a)] {\bf Complete Euler product decomposition} - The Euler product in \eqref{eprod} factorizes completely, i.e.,
    \begin{equation}\label{complete factorization}
        F(s) := \prod_p \prod_{j=1}^k \left(1 - \frac{\alpha_j(p)}{p^s}\right)^{-1}
    \end{equation}
    with $|\alpha_j|\leq 1$ and $\Re(s)>1$.
    \item[(b)] {\bf Zero-free region} - There exists a positive constant $c_F$, depending on $F$, such that $F(s)$ has no zeros in the region
    \begin{equation*}
        \Re(s) \geq 1 - \frac{c_F}{\log (|\Im(s)| + 2)},
    \end{equation*}
    except the possible Siegel-zero of $F(s)$.
    \item[(c)] {\bf Growth condition} - For $s=\sigma +it$, define
    \begin{equation*}
        \mu_F^*(\sigma) := \inf\{\lambda>0 : |F(s)| \ll (|t|+2)^{\lambda} \}.
    \end{equation*}
    Then, 
    \begin{equation*}
        \frac{\mu_F^*(\sigma)}{1-2\sigma}
    \end{equation*}
    is bounded for $\sigma<0$.
    
\end{enumerate}
\end{definition}

One expects $\mathbb{S}$ to satisfy conditions $(a)$ and $(b)$. In fact, the Riemann zeta-function, the Dirichlet $L$-functions, the Dedekind zeta-functions and the Rankin-Selberg $L$-functions can be all shown to satisfy conditions $(a)$ and $(b)$. Furthermore, for elements in $\mathbb{S}$ the growth condition $(c)$ is a consequence of the functional equation \eqref{fneq}. However, it is possible to have $L$-functions not obeying a functional equation to satisfy the growth condition. One can consider linear combination of elements in $\mathbb{S}$ to see this. A family of $L$-functions based on growth condition was introduced by V. K. Murty in \cite{Km} and the reader may refer to \cite{anup-thesis} for more details on this family. Also the Igusa zeta-function, and the zeta function of groups have Euler products but may not have functional equation, which is discussed in \cite{Saut}.

\subsection{\bf The Main Theorem}
In this paper, we produce a lower bound for large values of $L$-functions in $\mathbb{G}$ on the $1$-line. For a meromorphic function $F(s)$ having a pole of order $m$ at $s=1$, define
\begin{equation}\label{c_-m(F)}
    c_{-m}(F) = \lim_{s\to 1}\, (s-1)^m F(s).
\end{equation}
\begin{theorem}\label{main-theorem}
Let $F\in\mathbb{G}$ have non-negative Dirichlet coefficients $a_F(n)$ and a pole of order $m$ at $s=1$. Then, there exists a constant $C_F>0$ depending on $F(s)$ such that
\begin{equation*}
    \max_{t\in[\sqrt{T},T]} |F(1+it)| \geq e^{\gamma_F} ((\log_2 T + \log_3 T)^m - C_F),
\end{equation*}
where $\gamma_F=  m \gamma + \log c_{-m} (F)$ and $\gamma$ is the Euler-Mascheroni constant.
\end{theorem}
In the above theorem, since $a_F(n)\geq 0$, we clearly have $m\geq 1$. This is important because if $F$ has no pole at $s=1$, it is possible for $F(s)$ to grow very slowly on the 1-line. \\

As an immediate corollary, we get the following result for Dedekind zeta-functions $\zeta_K(s)$. Let $K/\mathbb{Q}$ be a number field. The Dedekind zeta-function $\zeta_K(s)$ is defined on $\Re(s)>1$ as
\begin{equation*}
    \zeta_K(s) := \sum_{\textbf{0}\neq \mathfrak{a} \subseteq \mathcal{O}_K} \frac{1}{(\mathbb{N} \mathfrak{a})^s }  = \prod_{\mathfrak{p}} \left(1-\frac{1}{(\mathbb{N} \mathfrak{p})^s} \right)^{-1},
\end{equation*}
where $\mathfrak{a}$ runs over all non-zero integral ideals and $\mathfrak{p}$ runs over all non-zero prime ideals of $K$. The function $\zeta_K(s)$ has an analytic continuation to the complex plane except for a simple pole at $s=1$. Furthermore, $\zeta_K$ satisfies properties $\textit{(a)}, \textit{(b)}, \textit{(c)}$ and therefore $\zeta_K \in \mathbb{G}$. Thus, by Theorem \ref{main-theorem}, we have

\begin{corollary}\label{large-value-dedekind-zeta}
For a number field $K$, there exists a constant $C_K>0$ depending on $K$ such that
\begin{equation*}
    \max_{t\in [\sqrt{T}, T]} |\zeta_K(1+it)| \geq e^{\gamma_K} (\log_2 T + \log_3 T - C_K),
\end{equation*}
where $\gamma_K = \gamma + \log \rho_K$, with $\rho_K$ being the residue of $\zeta_K(s)$ at $s=1$.
\end{corollary}

The $L$-function associated to the Rankin-Selberg convolution of any two holomorphic newforms $f$ and $g$, denoted by $L(s,f\times g)$, is in the Selberg class. Moreover, it can also be shown that $L(s, f \times g) \in \mathbb{G}$. Here $f$ and $g$ are normalized Hecke eigenforms of weight $k$. It is known that if $L(s, f\times g)$ has a pole at $s=1$, then $f=g$. Hence, from Theorem \ref{main-theorem}, we have the following.

\begin{corollary}\label{large-value-rankin-selberg}
For a normalized Hecke eigenform $f$, there exists a constant $C_f>0$ such that
\begin{equation*}
    \max_{t\in [\sqrt{T}, T]} |L(1+it, f\times f)| \geq e^{\gamma_f} (\log_2 T + \log_3 T - C_f),
\end{equation*}
where $\gamma_f = \gamma + \log \rho_f$, with $\rho_f$ being the residue of $L(s, f\times f)$ at $s=1$.
\end{corollary}

The result obtained in Theorem \ref{main-theorem} is a refined version of the bound established by Aistleitner-Pa\'{n}kowski \cite{Chris2}, which states that if $F(s)$ is in the Selberg class and satisfies the prime number theorem, namely,
\begin{equation*}
    \sum_{p\leq x} |a_F(p)| = \kappa \, \frac{x}{\log x} + \mathcal{O}\left(\frac{x}{\log^2 x}\right),
\end{equation*}
then for large $T$,
\begin{equation}\label{chris-lukasz}
    \max_{t\in [T, 2T]} |F(1+it)| = \Omega\left(\left(\log\log T\right)^{\kappa}\right).
\end{equation}

Furthermore, since we are assuming the zero-free region in $\mathbb{G}$, using \cite[Theorem 1]{Perelli}, we have $\kappa = m$. Hence, we get a slightly more refined result than \eqref{chris-lukasz}, but on a larger interval $[\sqrt{T},T]$.\\
 
 The poles of any element $F$ in the Selberg class $\mathbb{S}$ are expected to arise from the Riemann zeta-function. More precisely, if $F(s)$ has a pole of order $m$ at $s=1$, then $F(s)/\zeta(s)^m$ is expected to be entire and in $\mathbb{S}$. Thus, it is not surprising to expect the lower bound in Theorem \ref{main-theorem} to be of the order $(\log \log T)^m$.\\
 
 It is possible to generalize Theorem~\ref{main-theorem} to the Beurling zeta-function \cite{beurling} by constructing a suitable resonator over Beurling numbers instead of integers. However, this will carry us far afield from our current focus. Hence, we relegate it to future research.
 
 \section{\bf Mertens' theorem for the class $\mathbb{G}$}
 In 1874, Mertens \cite{mertens} proved the following estimate for truncated Euler-product of $\zeta(s)$, which is also known as Mertens' third theorem given by
 \begin{equation*}
     \prod_{p\leq x} \left(1-\frac{1}{p}\right)^{-1} = e^{\gamma} \log x + O(1),
 \end{equation*}
 where $\gamma$ denotes the Euler-Mascheroni constant. The analogue of Mertens' theorem for number fields was proved by Rosen \cite{Rosen}, who showed that
 \begin{equation*}
     \prod_{\mathbb{N}\mathfrak{P}< x} \left(1-\frac{1}{\mathbb{N}\mathfrak{P}}\right)^{-1} = \rho_K e^{\gamma} \log x + O(1),
 \end{equation*}
 where $\rho_K$ denotes the residue of $\zeta_K(s)$ at $s=1$. The Mertens theorem for the extended Selberg class satisfying conditions $(a)$ and $(b)$ was proved by Yashiro \cite{yashiro} in 2013. Following similar approach, one can establish Mertens' theorem for $\mathbb{G}$, where we replace the functional equation by the growth condition. However, Yashiro's paper \cite{yashiro} seems to be available only on arXiv. Hence, we include the proof for the sake of completeness.
 
 \begin{theorem}\label{generalized-mertens-thm}
 Let $F(s)\in\mathbb{G}$. Suppose that $F(s)$ has a pole of order $m$ at $s=1$ and $c_{-m}(F)$ be as in \eqref{c_-m(F)}. Then, for a constant $0< C_F \leq 1$,
 \begin{equation*}
     \prod_{p\leq x} \prod_{j=1}^{k} \left( 1-\frac{\alpha_j(p)}{p}\right)^{-1} = c_{-m}(F) e^{\gamma m} (\log x)^m \left(1+O\left(e^{-C_F \sqrt{\log x}}\right)\right).
 \end{equation*}
 \end{theorem}
 \begin{proof}
 We closely follow the method of Yashiro \cite{yashiro}. Denote by
 \begin{equation*}
     F(1;x) := \prod_{p\leq x} \prod_{j=1}^{k} \left( 1-\frac{\alpha_j(p)}{p}\right)^{-1}.
 \end{equation*}
 Let 
 \begin{equation*}
     \log F(s) = \sum_{n=1}^{\infty} \frac{b_F(n)}{n^s}.
 \end{equation*}
 By the complete Euler product \eqref{complete factorization}, we have $b_F(n)=0$ if $n\neq p^r$ and $b_F(n) \ll n^{\theta}$ for some $\theta<1/2$. Since
 \begin{equation*}
     b_F(p^r) = \frac{1}{r} \sum_{j=1}^k \alpha_j(p)^r,
 \end{equation*}
 we have $|b_F(p^r)| \leq k$. Write
 \begin{align}\label{Fx(1)-defn}
     \log F(1;x)  & = \sum_{p\leq x} \sum_{r=1}^{\infty} \frac{b_F(p^r)}{p^r}\nonumber \\
      &= \sum_{n\leq x} \frac{b_F(n)}{n} + \sum_{\sqrt{x} < p \leq x}  \sum_{p^r > x} \frac{b_F(p^r)}{p^r} + \sum_{p\leq \sqrt{x}} \sum_{p^r > x} \frac{b_F(p^r)}{p^r}.
 \end{align}
It is easy to estimate the second and third term above as follows.
 \begin{align*}
     \sum_{\sqrt{x} < p \leq x}  \sum_{p^r > x} \frac{b_F(p^r)}{p^r}  \ll \sum_{\sqrt{x} < p \leq x}  \sum_{r=2}^{\infty} \frac{1}{p^r} \ll \sum_{\sqrt{x} < p \leq x} \frac{1}{p^2} \ll \frac{1}{\sqrt{x}}.
 \end{align*}
 Also,
 \begin{align*}
     \sum_{p\leq \sqrt{x}} \sum_{p^r > x} \frac{b_F(p^r)}{p^r} & \ll \sum_{p\leq \sqrt{x}} \frac{1}{x} \ll \frac{1}{\sqrt{x}}.
 \end{align*}
From \eqref{Fx(1)-defn}, we get
 \begin{equation*}
     \log F(1;x) = \sum_{n\leq x} \frac{b_F(n)}{n} + O\left(\frac{1}{\sqrt{x}}\right).
 \end{equation*}
 Setting $(1/\log x)=u$ and $ e^{\sqrt{\log x}}=T$ and using Perron's formula, we get
 \begin{equation*}
     \sum_{n\leq x} \frac{b_F(n)}{n} = \frac{1}{2\pi i} \int_{u-iT}^{u+iT} \frac{x^s}{s}\, \log F(1+s) \, ds + O\left(e^{-c_F\sqrt{\log x}}\right).
 \end{equation*}

Let $u' = C_F / \log T = C_F / \sqrt{\log x}$. Choosing $x$ sufficiently large, we can ensure that there are no Siegel zeros for $F(1+s)$ in the region $[-u', u]$. Hence from the condition $(b)$, $F(1+s)$ has no zeros in the region $-u' \leq \Re(s) \leq u$ and $|\Im(s)| \leq T$ and has a pole of order $m$ at $s=0$.\\

Consider the contour $C$ joining $u-iT, -u'-iT, -u'+iT$ and $u+iT$. By the residue theorem, we have
\begin{equation}\label{residue_at_0}
    \text{Res}_{s=0} \left(\frac{x^s}{s} \log F(1+s)\right) = \frac{1}{2\pi i} \int_C \frac{x^s}{s} \log F(1+s) \, ds.
\end{equation}
We now estimate the above integral. Suppose $s=\sigma +it$. By the growth condition $(c)$, we have
\begin{equation*}
    |F(s)| \ll |t|^{\mu_F(\sigma)},
\end{equation*}
where $\mu(\sigma) \ll (1-2\sigma)$. Thus, for our choice of $u$ and $u'$, we get for $\sigma \in [-u',u]$
\begin{equation*}
    \log F(1+\sigma + iT) \ll (\log T)^2. 
\end{equation*}
Hence, we have
\begin{align}\label{estimate1}
    \left | \frac{1}{2\pi i} \int_{u+iT}^{-u'+iT} \frac{x^s}{s} \log F(1+s) \, ds \right | & \ll \left| \frac{(\log T)^2}{T} \int_{u}^{-u'} x^{\sigma} \, d\sigma \right |\nonumber\\
    & \ll (\log x) e^{-\sqrt{\log x}}\nonumber\\
    & \ll e^{-c_F' \sqrt{\log x}},
\end{align}
for some $0< c_F'< 1$. Similarly, we also get
\begin{equation}\label{estimate2}
     \left | \frac{1}{2\pi i} \int_{-u'+iT}^{u+iT} \frac{x^s}{s} \log F(1+s) \, ds \right | \ll e^{-c_F' \sqrt{\log x}}.
\end{equation}

We use the following result due to Landau (see \cite[p. 170, Lemma 6.3]{Vaughan}) to esimate the other terms in \eqref{residue_at_0}.
\begin{lemma}\label{zeros-lemma-titchmarsh}
Let $f(z)$ be an analytic function in the region containing the disc $|z|\leq 1$, supposing $|f(z)|\leq M$ for $|z|\leq 1$ and $f(0)\neq 0$. Fix r and R such that $0<r<R<1$. Then, for $|z|\leq r$ we have
\begin{equation*}
    \frac{f'}{f}(z) = \sum_{|\rho|\leq R} \frac{1}{z-\rho} + O\left(\log \frac{M}{|f(0)|}\right),
\end{equation*}
where $\rho$ is a zero of $f(s)$.
\end{lemma}
\medskip

Let $f(z) = (z+1/2+it)^m F(1+z+(1/2+it))$, $R=5/6$ and $r=2/3$ in the above Lemma \ref{zeros-lemma-titchmarsh}. Using the zero-free region $(b)$, we get
\begin{equation*}
    \left|\log s^m F(1+s)\right| \ll \left\{
	\begin{array}{ll}
		\log(|t| + 4),  & |t|\geq 7/8 \mbox{ and } \sigma \geq -u', \\
		1 & |t|\leq 7/8 \mbox{ and } \sigma \geq -u'.
	\end{array}
\right.
\end{equation*}
We now have the estimate
\begin{align}\label{estimate3}
    \left| \int_{-u'}^{-u'+iT}  \frac{x^s}{s} \log F(1+s) \, ds \right| & \ll \int_0^T \frac{x^{-u'}}{|s|} ( |\log s^m| + |\log s^m F(1+s)| ) \, dt\nonumber\\
    & \ll e^{-c_F'' \sqrt{\log x}},
\end{align}
for some $0<c_F'' <1$. Similarly, we also have
\begin{equation}\label{estimate4}
    \left| \int_{-u'-iT}^{-u'}  \frac{x^s}{s} \log F(1+s) \, ds \right| \ll e^{-c_F'' \sqrt{\log x}}.
\end{equation}
Using the estimates \eqref{estimate1}, \eqref{estimate2}, \eqref{estimate3} and \eqref{estimate4} in the Equation \eqref{residue_at_0} and choosing $C_F = \min(c_F,c_F',c_F'')$, we get
\begin{equation*}
    \frac{1}{2\pi i} \int_{u-iT}^{u+iT}  \frac{x^s}{s} \log F(1+s) \, ds = \text{Res}_{s=0} \left(\frac{x^s}{s} \log F(1+s)\right) + O\left(e^{-C_F \sqrt{\log x}}\right)
\end{equation*}
Let $\mathcal{C}$ denote the circle of radius $u'$ centered at $0$. Then,
\begin{equation*}
\frac{1}{2\pi i} \int_{\mathcal{C}} \frac{x^s}{s} \log F(1+s) \, ds = \text{Res}_{s=0} \left(\frac{x^s}{s} \log F(1+s)\right).
\end{equation*}
Hence, it suffices to estimate the above integral. Since $F(s)$ has a pole of order $m$ at $s=1$,
\begin{equation*}
    c_{-m}(F) = \lim_{s\to 1} (s-1)^m F(s) \neq 0.
\end{equation*}
Writing $F(s+1) = (s^{-m}) (s^m F(s+1))$, we get
\begin{equation}\label{main-integral}
    \frac{1}{2\pi i} \int_{\mathcal{C}} \frac{x^s}{s} \log F(1+s) \, ds = -\frac{m}{2\pi i} \int_{\mathcal{C}} \frac{x^{s}}{s} \log s\, ds + \log c_{-m}(F).
\end{equation}
The integral on the right hand side is
\begin{align}\label{integral-1}
    \int_{\mathcal{C}} \frac{x^{s}}{s} \log s\, ds &= \int_{-\pi}^{\pi} \frac{x^{u'e^{i\theta}}}{u'e^{i\theta}} (\log u'e^{i\theta}) (iu'e^{i\theta}) \, d\theta \\
    & = i(\log u')\int_{-\pi}^{\pi} e^{u' e^{i\theta}\log x} d\theta - \int_{-\pi}^{\pi} \theta e^{u' e^{i\theta}\log x} \, d\theta.
\end{align}
 By the series expansion of exponential function, we have
 
 \begin{align*}
     \int_{-\pi}^{\pi} e^{u' e^{i\theta}\log x} d\theta &=   \int_{-\pi}^{\pi} d\theta + \sum_{r=1}^{\infty} \frac{(u'\log x)^r}{r!} \int_{-\pi}^{\pi} e^{ir\theta} d\theta\\
     & = 2\pi
 \end{align*}
Similarly,
\begin{align*}
    \int_{-\pi}^{\pi} \theta e^{u' e^{i\theta}\log x} \, d\theta &= \int_{-\pi}^{\pi} \theta d\theta + \sum_{r=1}^{\infty} \frac{(u'\log x)^r}{r!} \int_{-\pi}^{\pi} \theta e^{ir\theta}\, d\theta \\
    & = \sum_{r=1}^{\infty} \left(\frac{(u'\log x)^r}{r!}\right) \left(\frac{(-1)^r 2\pi}{ir}\right)\\
    & = \frac{2\pi}{i} \sum_{r=1}^{\infty} \frac{(-1)^r}{r!} \int_0^{u'\log x} w^{r-1} dw\\
    & = \frac{2\pi}{i}\int_0^{u'\log x} \frac{e^{-w} -1}{w}\, dw.
\end{align*}
But the Euler-Mascheroni constant $\gamma$ satisfies the identity
 \begin{equation*}
     \gamma = \int_0^1 \frac{1-e^{-w}}{w} dw - \int_{1}^{\infty} \frac{e^{-w}}{w} dw.
 \end{equation*}
Thus, we have
\begin{align}\label{integral-4}
    \int_0^{u'\log x} \frac{e^{-w} -1}{w}\, dw &= \gamma + \int_1^{u'\log x} \frac{dw}{w} - \int_{u'\log x}^{\infty} \frac{e^{-w}}{w}\, dw\nonumber\\
    & = \gamma + \log \log x + \log u' + O\left(e^{-C_F\sqrt{\log x}}\right).
\end{align}
Combining the estimates above \eqref{integral-1}-\eqref{integral-4}, we get
\begin{equation*}
    \log F(1;x) = \log c_{-m}(F) + m\gamma + m\log \log x + O\left(e^{-C_F\sqrt{\log x}}\right).
\end{equation*}
Taking exponential on both sides and using the fact that $e^y = 1+O(y)$ for $|y|<1$, we are done.

 \end{proof}
\section{\bf Proof of the main theorem}

For $F\in\mathbb{G}$, define
\begin{equation*}
    F(s;Y) := \prod_{p\leq Y} \prod_{j=1}^k \left(1 - \frac{\alpha_j(p)}{p^s}\right)^{-1}.
\end{equation*}
We use the following approximation lemma.

\begin{lemma}\label{approximation-lemma}
For large $T$,
\begin{equation*}
    F(1+it) = F(1+it;Y) \left( 1 + O\left(\frac{1}{(\log T)^{10}}\right)\right),
\end{equation*}
for $Y= \exp((\log T)^{10})$ and $T^{1/10} \leq |t| \leq T$.
\end{lemma}
\begin{proof}
From the Euler product of $F(s)$, we have for $\Re(s)>1$,
\begin{align*}
\log F(s)=-\sum_{p}\sum_{j=1}^k\log\left(1-\frac{\alpha_j(p)}{p^{s}}\right)=\sum_{p}\sum_{j=1}^k\sum_{l}\frac{\alpha_j(p)^l}{lp^{ls}}.
\end{align*}
Let $t_0>0$ and let $\alpha>0$ be any sufficiently large constant.
Define 
\[\sigma_0:=\frac{1}{\alpha \log T}, \ \sigma_1:=\frac{1}{(\log T)^{20}} \ \text{ and } \ T_0:=\frac{T^{1/10}}{2}.\]
Applying Perron's summation formula as in \cite[Theorem~II.2.2]{tenen}, we get
\begin{equation*}\label{eq:perronfnl}
\int_{\sigma_1-iT_0}^{\sigma_1+iT_0}\log F(1+it_0+s)\frac{Y^s}{s} ds=-\sum_{p\leq Y}\sum_{j=1}^k\log\left(1-\frac{\alpha_j(p)}{p^{1+it_0}}\right)+O\left(\frac{1}{(\log T)^{10}}\right).
\end{equation*}
Now, we shift the path of integration to the left. By the zero-free region of $F\in\mathbb{G}$, the only pole of the above integrand in $\Re(s)\geq \sigma_0$ and $\Im(s)\leq T_0$ is at $s=0$. Therefore, we have 
\begin{equation}\label{eq:perronfnl2}
\log F(1+it_0)=-\sum_{p\leq Y}\sum_{j=1}^k\log\left(1-\frac{\alpha_j(p)}{p^{1+it_0}}\right)+O\left(\frac{1}{(\log T)^{10}} + \int_{\mathcal C}\log F(1+it_0+s)\frac{Y^s}{s} ds\right),    
\end{equation}
where $\mathcal{C}$ is the contour joining $\sigma_0-iT_0, \sigma_1 - iT_0, \sigma_1 +iT_0$ and $\sigma_0+iT_0$. Since, $|\log F (\sigma+ it)|\ll \log t$ on $\mathcal C$, we get
\begin{align}
&\int_{\sigma_1-iT_0}^{-\sigma_0-iT_0}\log F(1+it_0+s)\frac{Y^s}{s} \,ds \ll \frac{\log T}{T^{1/10}},\hspace{2mm} \int_{-\sigma_0+iT_0}^{\sigma_1 + iT_0}\log F(1+it_0+s)\frac{Y^s}{s} ds \ll \frac{\log T}{T^{1/10}}, \label{eq:perron_error_one}
\end{align}
and
\begin{align}
\int_{-\sigma_0-iT_0}^{-\sigma_0+iT_0}\log F(1+it_0+s)\frac{Y^s}{s} ds
\ll \frac{(\log T)^2}{\exp{\left(\frac{1}{\alpha}(\log T)^9\right)}},\label{eq:perron_error_two}
\end{align}
where all implied constants are absolute. Substituting the bounds from (\ref{eq:perron_error_one}) and (\ref{eq:perron_error_two}) in (\ref{eq:perronfnl2}),
for $T^{1/10}\leq t_0\leq T$, we obtain
\[\log F(1+it_0)=-\sum_{p\leq Y}\sum_{j=1}^k\log\left(1-\frac{\alpha_j(p)}{p^{1+it_0}}\right)+O\left(\frac{1}{(\log T)^{10}}\right).\]
Similarly we may argue when $t_0$ is negative. 
\end{proof}
By Lemma \ref{approximation-lemma}, it suffices to show Theorem \ref{main-theorem} for $F(1+it; Y)$. We closely follow the argument in \cite{Chris}. Set
\begin{equation*}
    X = \frac{1}{6} (\log T) (\log_2 T)
\end{equation*}
and for primes $p\leq X$ set
\begin{equation*}
    q_p = \left( 1-\frac{p}{X} \right).
\end{equation*}
Also set $q_1 = 1$ and $q_p=0$ for $p>X$. Extend the definition completely multiplicatively to define $q_n$ for all integers $n\geq 1$. Now define
\begin{equation*}
    R(t) = \prod_{p\leq X} (1- q_p p^{it})^{-1}.
\end{equation*}
Then we have
\begin{align*}
    \log (|R(t)|) & \leq \sum_{p\leq X} (\log X - \log p)\\
    &= \pi(X) \log X - \vartheta(X),
\end{align*}
 where $\pi(X)$ is the prime counting function and $\vartheta(X)$ is the first Chebyshev function. By partial summation, we know that 
 \begin{align*}
     \pi(X) \log X - \vartheta(X) = \int_2^X \frac{\pi(t)}{t} \, dt =(1+o(1)) \frac{X}{\log X}.
 \end{align*}
 By our choice of $X$, we get
 \begin{equation}\label{R(t)-bound}
     |R(t)|^2 \leq T^{1/3 + o(1)}.
 \end{equation}
From the Euler product, $R(t)$ has the following series representation
\begin{equation*}
    R(t) = \sum_{n=1}^{\infty} q_n n^{it},
\end{equation*}
and hence we get
\begin{align*}
    |R(t)|^2 = (\sum_{n=1}^{\infty} q_n n^{it})(\sum_{n=1}^{\infty} q_n n^{-it}) = \sum_{m,n=1}^{\infty} q_m q_n \left(\frac{m}{n}\right)^{it}.
\end{align*}
We have
\begin{equation*}
     F(1+it;Y) = \prod_{p\leq Y} \prod_{j=1}^k \left(1 - \frac{\alpha_j(p)p^{-it}}{p}\right)^{-1}
\end{equation*}
Since $|\alpha_j(p)|\leq 1$, we get
\begin{equation*}
      |F(1+it;Y)| \ll (\log Y)^k \ll (\log T)^{10k}.
\end{equation*}
Set $\Phi(t) := e^{-t^2}$ and recall that its Fourier transform is positive.
\medskip
Using \eqref{R(t)-bound}, we have
\begin{equation*}
    \left| \int_{|t|\geq T} F(1+it;Y) |R(t)|^2 \Phi \left(\frac{\log T}{T} t\right)\, dt \right| \ll 1,
\end{equation*}
and
\begin{equation*}
    \left| \int_{|t|\leq \sqrt{T}} F(1+it;Y) |R(t)|^2 \Phi \left(\frac{\log T}{T} t\right)\, dt \right| \ll T^{5/6 + o(1)}.
\end{equation*}

\medskip

Using the fact that $q_1=1$ and the positivity of the Fourier coefficients of $\Phi$, we also have the following lower bound

\begin{equation*}
    \int_{\sqrt{T}}^T \,|R(t)|^2\, \Phi\left(\frac{\log T}{T} t\right)\, dt \gg T^{1+o(1)}.
\end{equation*}
By a similar argument, again using the positivity of the Fourier coefficients, we have

\begin{equation*}
    \int_{-\infty}^{\infty} F(1+it; Y)|R(t)|^2 \Phi\left(\frac{\log T}{T} t\right)\, dt \geq \int_{-\infty}^{\infty} F(1+it; X)|R(t)|^2 \Phi\left(\frac{\log T}{T} t\right)\, dt.
\end{equation*}
So, we restrict ourselves to primes $p\leq X$ in the truncated Euler-product. This is to ensure both $R(t)$ and $F(1+it;X)$ have the terms with same $q$'s.\\

Write $F(1+it;X)$ as
\begin{equation*}
    F(1+it;X) := \sum_{n=1}^{\infty} a_k k^{-it},
\end{equation*}
where $a_k \geq 0$. This is because the Dirichlet coefficients of $F(s)$ are non-negative. Now define
\begin{align*}
    I_1 &:= \int_{-\infty}^{\infty} F(1+it; X)|R(t)|^2 \Phi\left(\frac{\log T}{T} t\right)\, dt\\
    &= \sum_{k=1}^{\infty} a_k \sum_{m,n=1}^{\infty} \int_{-\infty}^{\infty} k^{-it} q_m q_n \left(\frac{m}{n}\right)^{it} \Phi\left(\frac{\log T}{T} t\right) \, dt.
\end{align*}
We also define
\begin{equation*}
    I_2 := \int_{-\infty}^{\infty} |R(t)|^2 \Phi\left(\frac{\log T}{T} t\right)\, dt.
\end{equation*}
Notice that since we are working with truncated Euler-products, everything is absolutely convergent. Now, using the fact that the Fourier coefficients of $\Phi$ are positive and that $q_n$ are completely multiplicative, we get the inner sum of $I_1$ as

\begin{align*}
    \sum_{m,n=1}^{\infty} \int_{-\infty}^{\infty} k^{-it} q_m q_n \left(\frac{m}{n}\right)^{it} \Phi\left(\frac{\log T}{T} t\right) \, dt 
    & \geq \sum_{n=1}^{\infty} \sum_{k \mid m} \int_{-\infty}^{\infty} k^{-it} q_m q_n \left(\frac{m}{n}\right)^{it} \Phi\left(\frac{\log T}{T} t\right) \, dt \\
    & = q_k \sum_{n=1}^{\infty} \sum_{r=1}^{\infty} \int_{-\infty}^{\infty} q_r q_n \left(\frac{r}{n}\right)^{it} \Phi\left(\frac{\log T}{T} t\right)\, dt.
\end{align*}
Thus, we have
\begin{align}\label{I1/I2-bound}
    \frac{I_1}{I_2} & \geq \sum_{k=1}^{\infty} a_k q_k\nonumber = \prod_{p\leq X} \prod_{j=1}^k \left(1-\frac{\alpha_j(p)}{p} q_p \right)^{-1}\nonumber\\
    & = \left(\prod_{p\leq X} \prod_{j=1}^k \left( 1- \frac{\alpha_j(p)}{p} \right)^{-1}\right) \, \left(\prod_{p\leq X} \prod_{j=1}^k \left( \frac{p- \alpha_j(p)}{p-\alpha_j(p) q_p}\right)\right)
\end{align}
Using the generalized Merten's Theorem \ref{generalized-mertens-thm}, we have 
\begin{align}\label{generalized-mertens}
    \prod_{p\leq X} \prod_{j=1}^k \left( 1- \frac{\alpha_j(p)}{p} \right)^{-1} & = e^{\gamma_F} (\log X)^m + O(1)\nonumber\\ 
    &= e^{\gamma_F} (\log_2 T + \log_3 T)^m + O(1)
\end{align}
The second product in \eqref{I1/I2-bound} can be bounded as follows.
\begin{align}\label{second-product}
    -\log \left(\prod_{p\leq X} \prod_{j=1}^k \left( \frac{p- \alpha_j(p)}{p-\alpha_j(p) q_p}\right)\right) &= - \left(\sum_{p\leq X} \sum_{j=1}^k \log\left( \frac{p- \alpha_j(p)}{p-\alpha_j(p) q_p}\right)\right)\nonumber\\
    & \ll \sum_{p\leq X} \frac{1}{X}\\
    &\ll \frac{1}{\log X}.
\end{align}
From \eqref{I1/I2-bound},\eqref{generalized-mertens} and \eqref{second-product}, we get
\begin{equation*}
    \frac{I_1}{I_2} \geq e^{\gamma_F} (\log_2 T + \log_3 T)^m + O(1).
\end{equation*}
In other words, we have
\begin{equation*}
    \frac{\left|\int_{\sqrt{T}}^{T} F(1+it; X)|R(t)|^2 \Phi\left(\frac{\log T}{T} t\right)\, dt \right|}{\int_{\sqrt{T}}^{T} |R(t)|^2 \Phi\left(\frac{\log T}{T} t\right)\, dt} \geq e^{\gamma_F} (\log_2 T + \log_3 T)^m + O(1).
\end{equation*}
Hence, we conclude
\begin{equation*}
    \max_{t\in[\sqrt{T},T]} |F(1+it)| \geq e^{\gamma_F} ((\log_2 T + \log_3 T)^m - C_F).
\end{equation*}

\end{document}